\documentclass[12pt,leqno]{amsart}
\usepackage{amsmath,amssymb,amsfonts}
\usepackage{eucal}
\usepackage{enumerate}
\usepackage{graphicx}

\newcommand \comment[1]{}           
\renewcommand \comment[1]{\emph{[#1]}}      

\newtheorem{lem}{Lemma}
\newtheorem{cor}[lem]{Corollary}
\newtheorem{prop}[lem]{Proposition}
\newtheorem{thm}[lem]{Theorem}
\newtheorem{problem}[lem]{Problem}

\theoremstyle{definition}

\newtheorem{defini}[lem]{Definition}

\renewcommand{\phi}{\varphi}                 
\renewcommand{\epsilon}{\varepsilon}

\newcommand\bbN{\mathbb{N}}

\newcommand\bbR{\mathbb{R}}
\newcommand\bbZ{\mathbb{Z}}

\newcommand\cA{{\mathcal A}}
\newcommand\cN{{\mathcal N}}

\newcommand\G{\Gamma}

\begin{document}

\begin{center}

\Large
{Bijections between affine hyperplane arrangements and valued graphs}

\vskip10pt
\normalsize

DEDICATED TO MICHEL LAS VERGNAS

\vskip20pt

{Sylvie Corteel\footnote{The research of the first  author is supported by the ICOMB project, grant number ANR-08-JCJC-0011.}  \\
LIAFA, CNRS et Universit\'e Paris Diderot 7   \\
 Case 7014,
75205 Paris cedex 13, France \\[5pt]

David Forge and V\'eronique Ventos\footnote{The research of the last two  authors is supported by the TEOMATRO project, grant number ANR-10-BLAN 0207.}\\
Laboratoire de recherche en informatique UMR 8623\\
B\^at.\ 650, Universit\'e Paris-Sud\\
91405 Orsay Cedex, France\\[5pt]

E-mail: {\tt corteel@liafa.univ-paris-diderot.fr, forge@lri.fr, ventos@lri.fr}}\\[10pt]

\end{center}

\small
 {\sc Abstract.}
We show new bijective proofs of previously known formulas for the number of regions of some deformations of the braid arrangement, 
by means of a bijection between the no-broken-circuit sets of the corresponding integral gain graphs and some kinds of labelled binary trees. 
This leads to new bijective proofs for the Shi, Catalan, and similar hyperplane arrangements. 


\emph{Mathematics Subject Classifications (2010)}:
{\emph{Primary} 05C22; \emph{Secondary} 05A19, 05C05, 05C30, 52C35.}

\emph{Key words and phrases}:
{Integral gain graph, no broken circuits,
local binary search tree, Shi arrangement, Braid arrangement, Affinographic hyperplane arrangement}
 \normalsize\\


\vskip 20pt

\section{Introduction}\label{intro}

An \emph{integral gain graph} is a graph whose edges are labelled
invertibly by integers; that is, reversing the direction of an edge
negates the label (the \emph{gain} of the edge).  The \emph{affinographic
hyperplane arrangement}, $\cA[\Phi]$, that corresponds to an integral gain graph $\Phi$
is the set of all hyperplanes in $\bbR^n$ of the form $x_j-x_i=g$ for
edges $(i,j)$ with $i<j$ and gain $g$ in $\Phi$.  (See \cite[Section IV.4.1, pp.\ 270--271]{BG} or \cite{SOA}.)

 In recent years there has been much interest in real hyperplane
arrangements of this type, such as the braid arrangement, the Shi arrangement, the Linial
arrangement, and the composed-partition or Catalan arrangement. For
all these families,  the characteristic
polynomials and the number of regions  have been found \cite{PS}. 
For the Shi, Braid, Linial and Catalan arrangements, it is known that the regions
are in bijection with certain labelled trees or parking functions. See \cite{A99,AL99,GL,GLV, P,St2}
and references therein.

In this paper we give bijective proofs of the number of regions for some of these arrangements by establishing bijections 
between the  no-broken-circuit (NBC) sets and types of labelled trees and forests, which can be counted directly.  
This means that we use the fact that the number of regions is equal to the number
of NBC sets. This idea allows us to give a bijection between regions of hyperplane arrangements
defined by $x_j-x_i=g$ with $g\in [a,b]$ and $a+b=0$ or $a+b=1$; 
that is the hyperplane arrangements of the type "extended braid" and "extended Shi".

The paper is organized as follows. In Section \ref{defs}, we give some basic definitions.
In Section 3, we define the core idea; that is the definition of the height function.
In Section 4, we present NBC sets and trees. In Section 5, we characterize the NBC trees
and use this characteriztion in Section 6 to give a bijection between the NBC-trees and certain
labelled trees. In Section 7, we highlight two special cases and we end in Section 8 with some
concluding remarks.


\section{Basic definitions}\label{defs}

An \emph{integral gain graph} $\Phi = (\G,\phi)$ consists of a graph
$\G=(V,E)$ and an orientable function $\phi: E \to \bbZ$, called the
\emph{gain mapping}.  Orientability means that, if $(i,j)$ denotes an edge
oriented in one direction and $(j,i)$ the same edge with the opposite
orientation, then $\phi(j,i) = -\phi(i,j)$.  We have no loops
but multiple edges are permitted. For the rest of the paper, 
we denote the vertex set by $V = \{1,2,\ldots,n\} =: [n]$ with $n\ge1$.
 We use the notations $(i,j)$ for an edge with endpoints $i$ and $j$, oriented from $i$ to $j$, and $g(i,j)$ for
such an edge with gain $g$; that is, $\phi(g(i,j))=g$.  (Thus $g(i,j)$ is the same edge as $(-g)(j,i)$.  The edge $g(i,j)$ corresponds to a hyperplane whose equation is $x_j-x_i=g$.)
 A \emph{circle} is a connected 2-regular subgraph, or its edge set.  
 Writing a circle $C$ as a word $e_1e_2\cdots e_l$, the gain of $C$ is
$\phi(C):=\phi(e_1)+\phi(e_2)+\cdots+\phi(e_l)$; then it is well defined whether the gain is zero or nonzero.
 A subgraph is called \emph{balanced} if every circle in it
has gain zero. We will consider most especially balanced circles.

Given a linear order $<_O$ on the set of edges $E$, a \emph{broken circuit} is the set
of edges obtained by deleting the smallest element in a balanced circle.  
A set of edges, $N\subseteq E$, is a \emph{no-broken-circuit set} (NBC set for short)
if it contains no broken circuit. This notion from matroid theory
(see \cite{Bjorner} for reference) is very important here. We denote by 
$\mathcal N$ the set of NBC sets of the gain graph. It is well known 
that this set depends on the choice of the order, but its cardinality does not.

We can now transpose some ideas or problems from hyperplane arrangements
to gain graphs. For any integers $a,b, n$, let $K_n^{ab}$ be the gain graph
built on vertices $V=[n]$ by putting on every edge $(i,j)$ all the gains
$k$, for $a\le k\le b$. These gain graphs are expansion of the complete graph and 
their corresponding arrangements are called sometimes  deformations of the braid 
arrangement, truncated arrangements or affinographic arrangements.  We have four main 
examples coming from well known hyperplane arrangements.  
We denote by $B_n$ the gain graph $K_n^{00}$  and call it the \emph{braid gain graph}, by
$L_n$ the gain graph $K_n^{11}$  and call it the 
\emph{Linial gain graph}, by $S_n$ the gain graph $K_n^{01}$  and call it the \emph{Shi gain graph}
and finally by $C_n$ the gain graph $K_n^{-11}$  and call it the \emph{Catalan gain graph}.

\section{Height}

We introduce the notion of height function on an integral gain graph on the vertex set $[n]$.
A height function $h$ defines two important things for the rest of the paper: the induced 
gain subgraph $\Phi[h]$ of a gain graph $\Phi$ and an order $O_h$ on the set
of vertices extended lexicographically to the set of edges.

\begin{defini}
A \emph{height function} on a set $V $ is a function $h$ from
$V$ to $\bbN$ (the natural numbers including 0) such that $h^{-1}(0)\not=\emptyset$. The \emph{corner} of the
height function is the smallest element of greatest height.

Let $\Phi$ be a connected and balanced integral gain graph on a set $V$ of integers. 
The {\it height function} of the gain graph is the unique height function
$h_\Phi$ such that for every edge $g(i,j)$ we have $h_\Phi(j)-h_\Phi(i)=g$.
(Such a function exists iff $\Phi$ is balanced.) The \emph{corner of $\Phi$} is the corner of $h_\Phi$.


We say that an edge $g(i,j)$ is \emph{coherent with $h$} if $h(j)-h(i)=g$.
\end{defini}


\begin{defini}
Let  $\Phi$ be a gain graph also on $V=[n]$ and $h$ be a height function on  $V$. The subgraph $\Phi[h]$ of $\Phi$ \emph{selected by $h$}
is the gain subgraph on the same vertex set $V$ whose edges are the edges of $\Phi$ that are coherent with $h$.
\end{defini}


\begin{defini}
Given a height function $h$ on the set $V$, the order $O_h$ on the set $V=[n]$ is defined by $i<_{O_h}j$ iff $h(i)>h(j)$ or  ($h(i)=h(j)$ and $i<j$). 
The order $O_h$ is extended lexicographically to an order $O_h$ on the set of edges coherent with the height function.
\end{defini}

For example if $n=4$, $a=0$, $b=1$, and the height function $h$ has $h(2)=h(4)=1$ and $h(1)=h(3)=0$, we get the order
$2<_{O_h} 4 <_{O_h}1 <_{O_h} 3$. The corresponding $K_{4}^{01}[h]$ is given in Figure \ref{example}. Note that only 5 of the 12
edges  are coherent with the height function.

\begin{figure}
\includegraphics{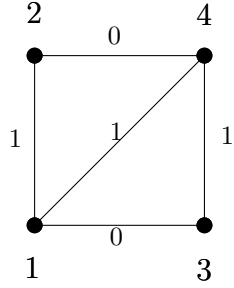}
\caption{The gain graph $K_{4}^{01}[h]$ for $h(2)=h(4)=1$ and $h(1)=h(3)=0$}
\label{example}
\end{figure}

\section{NBC sets and NBC trees in  gain graphs}\label{nbc-semi}

Given a linear order $<_O$ on the set of edges $E$, a \emph{broken circuit} is the set
of edges obtained by deleting the smallest element in a balanced circle.  
An NBC set in a gain graph $\Phi$ is basically an edge set, as it arises from matroid theory.  We usually assume an NBC set is a spanning subgraph, i.e., it contains all vertices.  Thus, an NBC tree is a spanning tree of $\Phi$.  Sometimes we wish to have non-spanning NBC sets, such as the components of an NBC forest; then we write of NBC \emph{subtrees}, which need not be spanning trees. The set of the NBC sets of $\Phi$ with respect to an order $O$
is denoted ${\mathcal N}_O(\Phi)$.

Given a height function $h$, a gain graph $\Phi$ and a linear order $<_{O_h}$ on the edges, they determine the set of NBC sets of the subgraph $\Phi[h]$ relative to the order $<_{O_h}$, 
denoted by ${\mathcal N}_O(\Phi[h])$. As always, this set depends on the choice of the order but  its cardinality does not.

\begin{lem}
Given an NBC tree $A$ of height function $h$ ($h=h_A$) with corner $c$, the forest $A\setminus c$ 
is a disjoint union of NBC subtrees 
of height functions $h_1$,...,$h_k$, and the orders $O_{h_i}$ are restrictions of the order $O_h$.\qed
\end{lem}

It is known from matroid theory that the NBC sets of the semimatroid of an affine arrangement $\cA$, 
with respect to a given ordering $<_O$ of the edges, correspond to the regions of the arrangement \cite[Section 9]{PS}.  
The semimatroid of $\cA[\Phi]$ is the frame (previously ``bias'' in \cite{BG}) semimatroid of $\Phi$, which consists of the balanced edge sets of the gain graph 
$\Phi$ (\cite[Sect.\ II.2]{BG} or \cite{SOA}).  Thus, the NBC sets of that semimatroid are spanning forests of $\Phi$.  Therefore $|\cN_O(\Phi)|$ equals the number of regions of $\cA[\Phi]$.

We show that the total number of NBC trees in an integral gain graph $\Phi$ equals the sum, over all height functions $h$, of the number of NBC trees in $\Phi[h]$.

Let $\Phi$ be connected.  Then we can decompose $\cN_O(\Phi)$ into disjoint subsets $\cN_O(\Phi[h])$, one for each height function $h$ that is coherent with $\Phi$ (that means that $\Phi[h]$ is also connected).  We have now: 
\[
\cN_O(\Phi) = \biguplus_h \{ \cN_O(\Phi[h]) \mid h \text{ is coherent with } \Phi \}.
\]
Therefore, 
the total number of NBC trees of all $\Phi[h]$ with respect to all possible height functions $h$ equals the number of NBC trees of $\Phi$.

For example, the NBC trees corresponding to the gain graph $K_4^{01}[h]$ from Figure \ref{example}
are given in Figure \ref{example1}.
\begin{figure}
\includegraphics{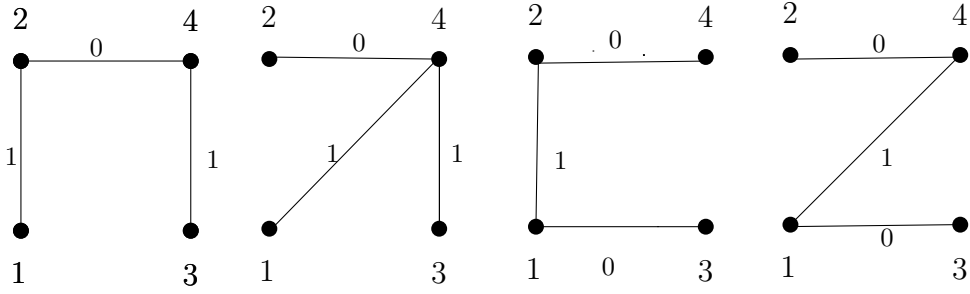}
\caption{The NBC trees of the gain graph $K_{4}^{01}[h]$}
\label{example1}
\end{figure}

\section{$[a,b]$-gain graphs and their NBC trees}

Let $a$ and $b$ be two integers such that $a\le b$. 
The interval $[a,b]$ is the set $\{i\in \bbZ \mid a\le i\le b\}$.
We consider the gain graph $K^{ab}_n$ with vertices labelled by $[n]$
and with all the edges $g(i,j)$, such that $i<j$ and $g\in [a,b]$. 
These gain graphs, $K^{ab}_n$, are called $[a,b]$-gain graphs.
The arrangements that correspond to these
gain graphs, called deformations of the braid arrangement, have been of particular interest.  
The braid arrangement corresponds to the special case $a=b=0$. 
Other well studied cases are $a=-b$ (extended Catalan), $a=b=1$
(Linial) and $a=b-1=0$ (Shi).

We will describe the set of NBC trees of $K^{ab}_n[h]$ for a given
height function $h$.
The idea is that, as mentioned above, the height function $h$ defines
an order $O_h$ on a balanced subgraph. We will then be able to describe the NBC sets
coherent with $h$ for the order $O_h$.


\begin{prop}
Let $a$ and $b$ be integers such that $a\le b$.  
Let $h$ be a height function of corner $c$ and let $\Phi$ be a spanning tree of 
$K^{ab}_n[h]$.  Suppose $c$ is incident to
the edges $g_i(c,v_i)$, $1\le i\le k$, and let $\Phi_i$ be the  connected
component of $\Phi\setminus c$ containing $c_i$ (which is a subtree). 
Then  $\Phi$   is an NBC tree if and only if
all the $\Phi_i $ are NBC trees and each $v_i$ is the $O_h$-smallest  vertex of $\Phi_i$ adjacent to $c$ in $K^{ab}_n[h]$.
\label{david}
\end{prop}
\begin{proof}
Everything comes from the choice of the order $O_h$ for the
vertices and the edges. If we have a vertex $v$ in $\Phi_i$
such that $v<_{O_h}v_i$ for which the edge $(c,v)\in K^{ab}_n[h] $ exists then this edge 
is smaller than all the edges of $\Phi _i +c$. Such an edge then closes a balanced
circle being the smallest edge of the circuit which is not possible.

In the other direction, if $\Phi$ is not an NBC tree then there is an 
edge $(x,y)$ in $ K^{ab}_n[h] $ closing a balanced circle by being the
smallest edge of the circuit. Since the $\Phi_i$ are by hypothesis are NBC trees
the vertices $x$ and $y$ cannot be in the same $\Phi_i$. They cannot be
in two different $\Phi_i$ either since the smallest edge would contain
$c$ necessarily. The last solution is that one of the vertex, say $x$, is $c$
and that the other vertex $y$ is in a $\Phi_i$. Since the edge $(c,v_i)$
will be in the circuit we need to have  $(x,y)<_{O_h}(c,v_i)$.
This implies the condition of the proposition.
\end{proof} 

\section{$[a,b]$-gain graphs with $a+b=0$ or 1}




We start this Section by a Lemma that will help us for our recursive construction. 
\begin{lem}
If $a+b=0$ or 1, the vertices $v_i$ are the 
corners of the subtrees $\Phi_i$ (as in Proposition \ref{david}). 
\label{lemmetec}
\end{lem}
\begin{proof}
In the case where $a+b=0$, the interval $[a,b]$ is of the form $[-b,b]$ where $b$ is a nonnegative integer. Similarly in the case
where $a+b=1$ the interval $[a,b]$ is of the form $[-b,b+1]$ where $b$ is a positive integer. Therefore whenever 
a gain $g$ is present in the graph it implies that all gains in the interval $[-|g|+1,|g|]$ also exists in the
graph. 
Therefore, if there exists a vertex $v$  in $\Phi_i$   with $h(v)>h(v_i)$ then the edge $(c,v)$ necessarily exists  in $ K^{ab}_n[h]$. In the case $h(v)=h(v_i)$ and $v<v_i$, the edge $(c,v)$ also necessarily exists  
in $ K^{ab}_n[h]$.

Let us suppose that $v_i$ is not the corner of its tree. Then there exists $v$ such that $h(v)>h(v_i)$ or $h(v)=h(v_i)$ and $v<v_i$. By taking the edge $(c,v_i)$ along with the unique path  $P(v_i,v)$ 
in this subtree we get a path $P$ which is a broken circuit of $K_n^{ab}[h]$ (because the edge $(c,v)$ is smaller in the order $O_h$ than all the edges of $P$) and this contradicts the fact that 
the tree $\Phi$ is an NBC tree of $K_n^{ab}$. Using Proposition \ref{david}, we get a contradiction and $v_i$ has to be the cornerof $\Phi_i$. 
\end{proof}

Note that this will not be true as soon as $a+b=2$ as in the Linial case.
We now introduce our family of trees.


\begin{defini}
Let $\alpha$ and $\beta$ be natural integers (including 0).  An {\it $(\alpha,\beta)$-rooted labelled tree} with $n$ vertices is a rooted, labelled and weighted tree on the set of vertices  $[n]$,
 such that each edge of the tree, $(i,j)$ where $i$ is the parent
and $j$ the child, is weighted with an integer from
\begin{itemize}
\item the interval $[1,\alpha]$ if $i<j$ and 
\item the interval $[1,\beta]$ if $i>j$.
\end{itemize}
\end{defini}

Note that if one of the integers $\alpha$ or $\beta$ is equal to 0 then 
the corresponding interval is empty. This  just implies that such edges cannot exist. 
In the next theorem we go from the NBC trees of $K_n^{ab}$ to $(\alpha,\beta)$-trees
by cutting the interval $[a,b]$ in two parts : the part $[a,0]$ of the negative 
or null gains will correspond to $\alpha$ and the part $[1,b]$ of the positive gains will correspond to $\beta$. 

\begin{thm}
If $b+a=0$ or $b+a=1$, the NBC trees of $K_n^{ab}$  are in bijection with the $(1-a,b)$-trees on $[n]$.
\label{aa}
\end{thm}

\begin{proof}
We recursively decompose the NBC trees of $K_n^{ab}$. 
Let $\Phi$ be an NBC tree. Let $c$ be its corner
and let $c_1,c_2,\ldots,c_k$ be the neighbors of $c$ with gains
$g_1,g_2,\ldots , g_k$. 
We now construct a corresponding $(1-a,b)$-tree. The root of the $(1-a,b)$-tree is $c$,
$c_1,c_2,\ldots ,c_k$ are its children and the edges from $c$ to $c_i$
get the label $g_i$ if it is strictly positive and $1-g_i$ otherwise. 
The decomposition continues recursively on the trees with corners $c_1,c_2,\ldots, c_k$.

When we take out the vertex $c$ from $\Phi$, we get a forest
of NBC trees, where each $c_i$ is in a different tree. 
To prove that the decomposition is correct, we use Lemma \ref{lemmetec} and we know that 
each $c_i$ is the corner of its component.
\end{proof}

A direct consequence of our Theorem \ref{aa} is that~:

\begin{cor}
If $b+a=0$ or $b+a=1$, the number of regions of $\cA[K_n^{ab}]$ is equal to the number
of $(1-a,b)$-rooted labelled forests with $n$ vertices.
\end{cor}

\begin{proof}
To get this consequence from the previous theorem, we use the facts that 
for any affine hyperplane arrangement the number 
of NBC sets is equal to the number of regions 
and that an NBC set is a union of NBC trees. See Proposition 9.4 of \cite{PS}.
\end{proof}

Following \cite{PS}, section 4, we can refine our result. The Poincar\'e polynomial ${\rm Poin_{\mathcal A}}(q)$
of an arrangement ${\mathcal A}$ is a $q$-analogue of the number of regions of the hyperplane arrangement.
Thanks to the NBC theorem (Theorem 4.5 in \cite{PS}), it is also a $q$-analogue of the
number of NBC forests. To be precise if $P_{n,j}$ is the number of NCB forests with $j$ edges, we have
\begin{equation}
 {\rm Poin_{\mathcal A}}(q)=\sum_{j\ge 1}P_{n,j}q^{j}
 \label{poincare}
\end{equation}
The characteristic polynomial ${\rm \chi_{\mathcal A}}(q)$ of the hyperplane arrangement
${\mathcal A}$ is such that~:
$$
{\chi_{\mathcal A}}(q)=q^n{\rm Poin_{\mathcal A}}(-1/q)
$$
Let $f_{n,j}$ be the number of $(1-a,b)$-labelled forests with $n$ vertices and $j$ trees. We define
the generating polynomial $F_{n,1-a,b}(q)=\sum_{j}(-1)^{n-j}f_{n,j} q^{j-1}$.
We then get 
\begin{cor}
If $b+a=0$ or $b+a=1$, the characteristic polynomial ${\rm \chi_{\mathcal A}}(q)$ of $\cA[K_n^{ab}]$ is equal to $F_{n,1-a,b}(q)$ the generating polynomial
of $(1-a,b)$-rooted labelled forests with $n$ vertices.
\end{cor}

We will now give an alternative proof of Theorem 9.8 and Example 9.10.2 of \cite{PS}.
Note that our results look different from those in \cite{PS}, as Postnikov and Stanley look at the hyperplane arrangements
defined by $x_i-x_j=g$ with $g\in[-a+1,b-1]$  and such that $a\ge0$ and $b\ge -a+2$. We chose $x_i-x_j=g$ with $g\in[a,b]$
and such that $b\ge a$, $b \ge 0$ and $a\le 1$.
\begin{thm}
The characteristic polynomial ${\rm \chi_{\mathcal A}}(q)$ of $\cA[K_n^{ab}]$ is
\begin{align*}
&(-1)^{n-1}(bn-q+1)(bn-q+2)\ldots (bn-q+n-1-q), 	\quad &\text{ if } \ a+b=0,
\intertext{and}
&(-1)^{n-1}(bn-q)^{n-1}, 	\quad &\text{ if } \ a+b=1.
\end{align*}
\label{thmps}
\end{thm}

Setting $q=-1$ and taking the absolute value, we get that~: 
\begin{cor}
The number of regions of $\cA[K_n^{ab}]$ is
\begin{align*}
&(bn+2)(bn+3)\ldots (bn+n), 	\quad &\text{ if } \ a+b=0,
\intertext{and}
&(bn+1)^{n-1}, 	\quad &\text{ if } \ a+b=1.
\end{align*}
\end{cor}

To finish our proof of Theorem \ref{thmps}, we have to count the $(\alpha,\beta)$-labelled trees and $(\alpha,\beta)$-labelled forests. More general results on the enumeration of labelled trees can be found in \cite{GS,E}.

\begin{prop}
The number of $(\alpha,\beta)$-rooted labelled trees with $n$ vertices is
$$
\prod_{i=1}^{n-1}[n\beta+(\alpha-\beta)i].
$$
The generating polynomial $F_{n,\alpha,\beta}(q)$ of $(\alpha,\beta)$-rooted labelled forests with $n$ vertices is
$$
(-1)^{n-1}\prod_{i=1}^{n-1}[n\beta-q+(\alpha-\beta)i].
$$
\end{prop}
\begin{proof}
We suppose that $\alpha\ge \beta$. The other case is analogous. 
We  enumerate $(\alpha,\beta)$-rooted labelled forests. The statement on trees is straighforward by setting $q=0$. 
We suppose that the forest has $j$ trees, i.e., $n-j$ edges.
We split the edges of the trees into two groups~:
\begin{itemize}
\item The edges with labels $\beta+1,\ldots ,\alpha$.
\item The edges with labels $1,2,\ldots ,\beta$. 
\end{itemize}
Suppose that the first group has $k$ edges. They form a decreasing forest
on $n$ vertices with $k$ edges, such that the edges can have $(\alpha-\beta)$
different labels. The number of such forests is well known to be $|s(n,n-k)|(\alpha-\beta)^k$ where
$s(n,k)$ is the Stirling number of the first kind.

The second group is a rooted labelled  forest on $n$ vertices with $n-k-j$ edges, such that the edges can have $\beta$ different labels. 
The two groups have disjoint edges
and the $j$ non-existing edges are also disjoint. 
Therefore the generating polynomial of such forests on $n$ vertices is  $(n\beta)^{n-k-j}{n-k-1\choose j-1}q^{j-1}$.
As the forest has in total $n-j$ edges, we also need the sign $(-1)^{n-j}$.

Therefore, we deduce that the generating polynomial $F_{n,\alpha,\beta}(q)$  of $(\alpha,\beta)$-rooted forest trees with $n$ vertices and $k$ edges in the first group
is~:
$$
|s(n,n-k)|(\alpha-\beta)^k\sum_{j\ge 1}(1)^{n-j}(n\beta )^{n-k-j}{n-k-1\choose j-1}q^{j-1}=(-1)^{n-1}|s(n,n-k)|(\beta-\alpha)^k(n\beta-q)^{n-k-1}.
$$
Therefore the generating polynomial $F_{n,\alpha,\beta}(q)$ is~:
\begin{align*}
&(-1)^{n-1}\sum_{k=0}^{n-1} |s(n,n-k)|(\alpha-\beta)^k(n\beta-q)^{n-k-1}\\
&=(-1)^{n-1}\frac{(\alpha-\beta)^n}{n\beta}\sum_{k=0}^{n} |s(n,n-k)|\left(\frac{n\beta-q}{\alpha-\beta}\right)^{n-k}\\
&=(-1)^{n-1}\frac{(\alpha-\beta)^n}{n\beta-q}\prod_{i=0}^{n-1} \left(i+\frac{n\beta-q}{\alpha-\beta}\right)\\
&=(-1)^{n-1}\prod_{i=1}^{n-1}(n\beta-q+(\alpha-\beta)i).
\qedhere
\end{align*}
\end{proof}

\section{The special cases of the braid and the Shi arrangements}

The first cases of $[a,b]$-gain graphs with $a+b=0$ or $a+b=1$ are obtained by taking $a=0$. 
The gain graph with  $a+b=0$ and $a=0$ corresponds to the braid arrangement 
and the gain graph with  $a+b=1$ and $a=0$ corresponds to the Shi arrangement. 
A bijective correspondence for the braid arrangement, inducing activity preserving bijections between regions and NBC sets or increasing trees,
appears in the paper  \cite{GLV}.

\begin{cor}
The NBC sets of the braid arrangement in dimension $n$ are in one-to-one correspondence with 
the decreasing labelled trees on $n+1$ vertices.
\end{cor}

\begin{proof}
Theorem \ref{aa} tells us that the set of NBC trees of the braid arrangement (case $a=b=0$) 
is in one-to-one correspondence with the set of $(1,0)$-labelled trees with $n$ vertices. 
Such labelled trees have no possible value on edges $(i,j)$ when $i>j$ and have the value 1 on edges $(i,j)$ when $i<j$ (and since there is no choice we can forget the value). 
This means that the correspondence of NBC trees of the braid arrangement is with the set of rooted labelled trees such that the label of the father is always smaller than the label of the son (such a tree is called 
an increasing tree). 
To get the bijection between the set of NBC sets and the set of increasing rooted labelled
 trees we just need to add vertex $0$ and to connect it to  the different 
increasing rooted labelled trees coming from the NBC trees (components).
\end{proof}

For the Shi arrangement, Pak and Stanley \cite{St2} gave
a bijection between the regions and the parking functions. Lots of other bijections exist, as the regions of the
Shi arrangement are also in bijections with  labelled trees. 
See \cite{A99,AL99,GL} and references therein.

\begin{cor}
The NBC sets of the Shi arrangement in dimension $n$ are in one-to-one correspondence with 
the labelled trees on $n+1$ vertices.
\end{cor}

\begin{proof}
Theorem \ref{aa} tells us that the set of NBC trees of the Shi arrangement (case $a=0$ and  $b=1$) 
is in one-to-one correspondence with the set of $(1,1)$-labelled trees with $n$ vertices. 
Such labelled trees have the value 1 on edges $(i,j)$ when $i<j$ as well as when $j>i$. 
As in the previous proof, since there is only one possible value it can be ignored. 
This means that the correspondence of NBC trees of the Shi arrangement is with the set 
of  rooted labelled trees. To get the bijection between the set 
of NBC sets and the set of labelled trees on $n+1$ vertices we just need to add vertex $n+1$ and to connect the different  rooted labelled trees coming from the NBC trees (components).
\end{proof}

\section{Conclusion}

In this paper, we show that given a height function on a gain graph $K_n^{ab}$ with $a+b=0$ or 1, the corresponding NBC trees
with $n$ vertices and corner $c$ are in bijection with some trees with $n$ vertices and root $c$. In a forthcoming paper, we will show that this is still
true in the Linial case; that is, $a=0$ and $b=2$ \cite{forge}. 
We also investigate whether this is true for other deformations of the Braid arrangement \cite{CFM} and think that 
such constructions might also exist for hyperplane arrangements for root systems studied in \cite{A99}.\\

\noindent{\bf Acknowledgments.} The authors want to thank Tom Zaslavsky for his help during the elaboration
and the redaction of this work. We also thank Ira Gessel, Emeric Gioan and the anonymous referees for
their constructive comments on the paper and the bibliography.


\end{document}